\documentclass[12pt]{article}
\usepackage{color}
\usepackage{amsmath}
\usepackage{bbm}
\usepackage{extarrows}
\usepackage{geometry}
\usepackage{authblk} 
\usepackage{hyperref} 
\usepackage{color}
\usepackage{ntheorem}

\usepackage[mathscr]{euscript}
\usepackage[utf8]{inputenc}
\def\0{\emptyset}

\newtheorem{theorem}{Theorem}[section]

\newtheorem{lemma}[theorem]{Lemma}
\newtheorem{claim}[theorem]{Claim}

\newtheorem{cor}[theorem]{Corollary}

\newenvironment{reproof}{{\noindent\it Proof of Theorem \ref{thm: exodd when n large enough}.}}{\hfill $\square$\par}
\newenvironment{reproofcor}{{\noindent\it Proof of Theorem  \ref{cor: ex(c2k)}.}}{\hfill $\square$\par}
\newenvironment{reprooff}{{\noindent\it Proof of Theorem \ref{thm: exeven when n large enough}.}}{\hfill $\square$\par}

\newenvironment{proof}{{\noindent\it Proof.}}{\hfill $\square$\par}

\usepackage{graphicx}

\usepackage{enumerate}
\usepackage{enumitem}

\usepackage{
	amsmath,			
	amssymb,			
	enumerate,		    
	graphicx,			
	lastpage,			
	multicol,			
	multirow,			
	pifont,			    
}

\usepackage[numbers]{natbib}

\newcommand{\excoddm}{ex(n, K_r, \{C_{\geq 2k+1}, M_{s+1}\})}
\newcommand{\excevenm}{ex(n, K_r,\{C_{\geq 2k}, M_{s+1}\})}
\newcommand{\lf}{\left\lfloor}
\newcommand{\rf}{\right\rfloor}
\newcommand{\lc}{\left\lceil}
\newcommand{\rc}{\right\rceil}
\newcounter{cases}
\newcounter{subcases}[cases]
\newenvironment{mycase}
{
    \setcounter{cases}{0}
    \setcounter{subcases}{0}
    \newcommand{\case}
    {
        \par\indent\stepcounter{cases}\textbf{Case \thecases.}
    }
    
}
{
    \par
}
\renewcommand*\thecases{\arabic{cases}}

\begin{document}


\title{Generalized Tur{\'a}n problems for a matching  and  long cycles}
\author{

    {\small\bf Xiamiao Zhao}\thanks{email:  zxm23@mails.tsinghua.edu.cn}\quad
    {\small\bf Mei Lu}\thanks{Corresponding author: email: lumei@tsinghua.edu.cn}\\
    {\small Department of Mathematical Sciences, Tsinghua University, Beijing 100084, China.}\\
}

\date{}

\maketitle\baselineskip 16.3pt

\begin{abstract}
Let $\mathscr{F}$ be a family of graphs. A graph $G$ is $\mathscr{F}$-free if  $G$ does not contain any $F\in \mathcal{F}$  as a subgraph. The general Tur\'an number, denoted by $ex(n, H,\mathscr{F})$, is the maximum number of copies of $H$ in an $n$-vertex $\mathscr{F}$-free graph. Then $ex(n, K_2,\mathscr{F})$, also denote by $ex(n, \mathscr{F})$, is the Tur\'an number. Recently, Alon and Frankl
determined the exact value of $ex(n, \{K_{k},M_{s+1}\})$, where $K_{k}$ and $M_{s+1}$ are a
complete graph on $k $ vertices and a matching of size $s +1$, respectively. Then many results were obtained by extending $K_{k}$ to a general fixed
graph or family of graphs. Let $C_k$ be a cycle of order $k$. Denote $C_{\ge k}=\{C_k,C_{k+1},\ldots\}$. In this paper, we determine the  value of $ex(n,K_r, \{C_{\ge k},M_{s+1}\})$ for  large enough $n$ and obtain the extremal graphs when $k$ is odd. Particularly, the exact value of $ex(n, \{C_{\ge k},M_{s+1}\})$ and the extremal graph are given for  large enough $n$.
\end{abstract}


\vskip.3cm

\section{Introduction}
In this paper, we only consider finite, simple and undirected graphs. Let $G=(V,E)$ be a graph, where $V$ is the vertex set and $E$ is the edge set of $G$. We use $v(G)$ and $e(G)$ to denote the order and size of $G$, respectively.

Let $\mathscr{F}$ be a family of graphs. A graph $G$ is $\mathscr{F}$-free if  $G$ does not contain any $F\in \mathcal{F}$  as a subgraph. The general Tur\'an number, denoted by $ex(n, H,\mathscr{F})$, is the maximum number of copies of $H$ in an $n$-vertex $\mathscr{F}$-free graph. When $H=K_2$, we write $ex(n,\mathscr{F})$ instead of $ex(n,H,\mathscr{F})$. Here $ex(n,\mathscr{F})$ is the
classical Tur\'an number. If $\mathscr{F}=\{F\}$, we write $ex(n,H,F)$ instead of  $ex(n,H, \mathscr{F})$.  Denote
$$Ex(n,H, \mathscr{F}):=\{G~|~ G~\text{is a $n$-vertex $\mathscr{F}$-free graph with $ex(n,H, \mathscr{F})$ copies of $H$}\}.$$

Many scholars examined the maximum number of edges without a cycle of at least $k$. Erd\H{o}s and Gallai \cite{gallai1959maximal} showed that if a graph $G$ with $n$ vertices has more than $\frac{1}{2}(k-1)(n-1)$ edges, then $G$ contains a cycle of length at least $k$, where $k\geq 3$. But the upper bound is tight only when $k-2$ divides $n-1$.  Woodall \cite{woodall1972sufficient} completed all the rest cases by proving that if $k\geq 3$ and $n=q(k-2)+p+1$, where $q\geq 0$ and $0\leq p\leq k-2$, and $G$ is a graph on $n$ vertices with more than $q{k-1\choose 2}+{p+1\choose 2}$ edges, then $G$ contains a cycles of length at least $k$. This result is best possible because we can find a graph $G_0$ that is $C_{\geq k}$-free and has $q{k-1\choose 2}+{p+1\choose 2}$ edges, where  $G_0$ is a graph consisting of $q$ copies of $K_{k-1}$ and one copy of $K_{p+1}$, all having exactly one vertex in common. Noticing that the extremal graph above is not $2$-connected. 
Kopylov \cite{kopylov1977maximal} and Woodall \cite{woodall1976maximal} independently determined the maximum size of a 2-connected graph without cycles of length at least $k$. For $2\leq a\leq (k-1)/2$,  define
$$  f_b(n,k,a) := {k-a\choose b}+(n-k+a){a\choose b-1}.$$  Let $k\ge 2a$ and $H_{n,k,a}$ be a graph obtained from $K_{k-a}$ by adding $n- (k-a)$ isolated vertices each joined to the same $a$ vertices
of $K_{k-a}$. Then $  f_2(n,k,a)=e(H_{n,k,a})$. Note that when $a \ge  2$, $H_{n,k,a}$ is 2-connected and has no cycle of length at least $k$. Also the matching number of $H_{n,k,a}$ is $a+\lf\frac{k-2a}{2}\rf=\lf\frac{k}{2}\rf$.

\begin{theorem}[Kopylov \cite{kopylov1977maximal} and Woodall \cite{woodall1976maximal}]\label{thm: ex(n,Cgeq k)}
    Let $5\leq k\leq n$ and let $t=\lf\frac{k-1}{2}\rf$. If $G$ is a 2-connected $n$-vertices graph with
    $$e(G)\geq \max\{f_2(n,k,2),f_2(n,k,t)\},$$
    then either $G$ has a cycle of length at least $k$, or $G=H_{n,k,2}$, or $G=H_{n,k,t}.$
\end{theorem}

\begin{theorem}[Fan et al \cite{fan2004cycles}]\label{thm: ex(n,c) with k leq n-1}
    Let $2\leq k\leq n-1$ and let $t=\lf\frac{k-1}{2}\rf$. If $G$ is a 2-connected $n$-vertices graph with
    $$e(G)\geq \max\{f_2(n,k,2),f_2(n,k,t)\},$$
    then either $G$ has a cycle of length at least $k$, or $G=H_{n,k,2}$, or $G=H_{n,k,t}.$
\end{theorem}
 Since the matching number of $H_{n,k,a}$ is $\lf\frac{k}{2}\rf$, by Theorem \ref{thm: ex(n,c) with k leq n-1}, we have the following corollary.
\begin{cor}\label{cor: excm when G is 2-connected}
    Let $2\leq k\leq n-1$, $s\geq \lc \frac{k}{2}\rc$ and let $t=\lf\frac{k-1}{2}\rf$. If $G$ is a $2$-connected $n$-vertex graph with matching number at most $s$ and without cycles of length at least $k$, then
    $$e(G)\leq \max\{f_2(n,k,2),f_2(n,k,t)\}.$$The extremal graph is either $H_{n,k,2}$ or $H_{n,k,t}$.
\end{cor}
Luo \cite{luo2018maximum} extended the result of the Tur\'an problem on the cycle of length at least $k$ to the general Tur\'an problem. Let $N_r(G)$ denote the number of copies of cliques with size $r$ in $G$.
\begin{theorem}[Luo \cite{luo2018maximum}]\label{thm: 2-connected without Ck, count Kr}\label{thm: Luo eX(n,K_r,Ck)}

Let $n\geq k\geq 5$ and let $t = \lf \frac{k-1}{2}\rf$. If $G$ is a 2-connected $n$-vertex graph with circumference less than $k$, then
$$N_r(G)\leq \max\{f_r(n,k,2),f_r(n,k,t)\}.$$
The extremal graph is either $H_{n,k,2}$ or $H_{n,k,t}$.
\end{theorem}
 Similarly, we have the following result.
\begin{cor}\label{cor: ex(n,K_r) for 2-connected graph}
    Let $ k\geq 5$, $s\ge \lceil\frac{k}{2}\rceil$ and $t = \lf\frac{k-1}{2}\rf$. If $G$ is a 2-connected $n$-vertex graph with matching number at most $s$ and without cycles of length at least $k$, then
    $N_r(G)\leq f_r(n,k,t)$ when $n$ is  large enough. Moreover, the extremal graph is  $H_{n,k,t}$.
\end{cor}

Erd\H{o}s and Gallai \cite{erdos1961minimal} gave the result of the maximum graph with the matching number at most $s$, that is,
\begin{equation}
    ex(n,M_{s+1})=\max\left\{f_2(n,2s+1,s),{2s+1\choose 2}\right\}.
\end{equation}
Duan, Ning, Peng, Wang and Yang \cite{duan2020maximizing} extent Erd\H{o}s and Gallai's result.
\begin{theorem}[Duan et al.  \cite{duan2020maximizing}]\label{thm: duan thm}
    If $G$ is a graph with $n\geq 2s+2$ vertices, minimum degree $\delta$, and with matching number at most $s,$ then
    $$N_r(G)\leq\max \{f_r(n,2s+1,\delta),f_r(n,2s+1,s)\},$$
    i.e. $ex(n,K_r,M_{s+1}) =\max \{f_r(n,2s+1,\delta),f_r(n,2s+1,s)\}. $
\end{theorem}
Chakraborti and Chen \cite{chakraborti2020exact} gave the exact results about $ex(n,K_r,C_{\geq k})$. Dou, Ning and Peng \cite{dou2024number} determined the value of $ex(n,\{C_{\geq k},K_m\})$.

Recently, Alon and Frankl \cite{alon2024turan} proposed to consider the maximum number of edges in a $F$-free graph on $n$ vertices with matching number at most $s$ (denoted as $ex({n,\{F,M_{s+1}\}}$). In the same paper, they proved that
$$ex(n,\{K_{k+1},M_{s+1}\})=e(G(n,k,s)),$$
where $G(n,k,s)$ denotes the complete $k$-partite graph on $n$ vertices consisting of $k-1 $ vertex classes of sizes as equal as possible whose total size is $s$, and one additional vertex class of size $n-s$.
Later, Gerbner \cite{gerbner2024turan} gave several results about $ex({n,\{F,M_{s+1}\}})$,  when $F$ satisfies some properties. He proved that if $\chi(F)>2$ and $n$ is large enough, then $ex(n,\{F,M_{s+1}\})=ex(s,\mathscr{G}(F))+s(n-2)$, where $\mathscr{G}(F)$ is the family of graphs obtained by deleting an independent set from $F$. Luo, Zhao and Lu \cite{luo2024turannumbercompletebipartite} determined the exact value of $ex(n,\{K_{l,t},M_{s+1}\})$ when $s$ is large enough and $n\geq {3s\choose 2}$ for all $3\leq l\leq t$.
Ma and Hou \cite{ma2023generalized} determined the exact value of $ex(n,K_r,\{K_{k+1},M_{s+1}\})$ when $n\geq 2s+1$ and $k\geq r\geq 3$. Zhu and Chen \cite{zhu2023extremal} determined $ex(n,K_r,\{F,M_{s+1}\})$ for some fixed $F$. Gerbner \cite{gerbner2023generalized} gave several asymptomatic results about $ex(n,H,\{F,M_{s+1}\})$ for general graph $H$.

Let $C_k$ be a cycle of order $k$. Denote $C_{\ge k}=\{C_k,C_{k+1},\ldots\}$. In this paper, we determined the exact value of $ex(n,K_r, \{C_{\ge k},M_{s+1}\})$ for large enough $n$. Note that the matching number of $C_{k}$ is $\lfloor k/2\rfloor$. If $s\leq \lfloor\frac{k}{2}\rfloor-1$, then $ex(n,K_r, \{C_{\ge k},M_{s+1}\})=ex(n,K_r,M_{s+1})$ which has been determined.
So we consider the case when $s\geq \lc\frac{k}{2}\rc$. It is not difficult to show that $ex(n,K_2,\{C_{\geq 3},M_{s+1}\}) = n-1 $ and $K_{1,n-1}$ is the extremal graph. And $ex(n,K_r,\{C_{\geq 3},M_{s+1}\})=0$, for $r\geq 3$. So we can assume $k\ge 4$. Since the parity of $k$ causes different value of $ex(n,K_r, \{C_{\ge k},M_{s+1}\})$, we will describe our conclusions in two theorems.
Given integers $k$ and $r$, let $\tau_{k,r}=\min\{k_0\in \mathbb{N}~|~k_0{k\choose r-1}< {k_0+1\choose r}\}.$ Notice that ${a+b+1\choose r}\geq {a+1\choose r}+b{a\choose r-1}$ for every $b\geq 0$, $a\geq 1$, and $k_0{k\choose r-1}\geq k_0{k_0\choose r-1} \geq {k_0+1\choose r}$ for every $2\leq k_0\leq k$. So $\tau_{r,k}> k$. And we have $$(\tau_{k,r}+b){k\choose r-1}\leq \tau_{k,r}{k\choose r-1}+b{\tau_{k,r}\choose r-1}< {\tau_{k,r}+1\choose r}+b{\tau_{k,r}\choose r-1} \leq {\tau_{k,r}+b+1\choose r},$$ for every $b\geq 0$. Thus for every $k'\geq \tau_{k,r}$, $k'{k\choose r-1}<{k'+1\choose r}$. Moreover, above inequality implies when $k'\geq \tau_{r,k}$,
\begin{equation}\label{eq: the convexity of k choose r}
{k'+b+1\choose r}-(k'+b){k\choose r-1}\geq {k'+1\choose r}-k'{k\choose r-1}.\end{equation}
So the function $f(x)={x+1\choose r}-x{k\choose r-1}$ is monotonic increase about $x$ when $x\geq \tau_{k,r}$.

Let
$$h(r,k,s)=\left\{
\begin{array}{ll}
{k+1\choose r}-(k+1){k\choose r-1} & \mbox{if $2k\leq\tau_{k,r},$}\\
q{2k\choose r}+{k+1\choose r}-(k+1+q(2k-1)){k\choose r-1} & \mbox{if  $2t+1<\tau_{k,r}\le 2k-1$,}\\
q{2k\choose r}+{2t+2\choose r}+{k+1\choose r}-A{k\choose r-1} & \mbox{if  $2t+1\geq \tau_{k,r}$,}
\end{array}
\right.$$
where $A=k+1+q(2k-1)+(2t+1)$, $q=\lf\frac{s-k}{k-1}\rf$ and $t=s-k-q(k-1)<k-1$. It is not difficult to check that ${2k\choose r}-(2k-1){k\choose r-1}\ge 0$. When $2t+1\geq \tau_{k,r}$, we have ${2t+2\choose r}-(2t+1){k\choose r-1}\ge 0$. So $h(r,k,s)\ge {k+1\choose r}-(k+1){k\choose r-1}$. Now we have our first main result.


\begin{theorem}\label{thm: exodd when n large enough}
    If $k+1\geq r$, $k\geq 2,r\geq 2$  and $s\geq 2k+1$, then when $n$ is large enough,
    $$\excoddm = {k\choose r-1}n+ h(r,k,s).$$Moreover, when $2k\leq\tau_{k,r}$ (resp. $2t+1<\tau_{k,r}\le 2k-1$ or $2t+1\geq \tau_{k,r}$), the extremal graph is $ H_{n,2k+1,k}$ (resp.  $St(G,B_1,s_1)$ consisting of $q+1$ blocks with $B_1\cong H_{n-q(2k-1),2k+1,k}$ and $q$ blocks being $K_{2k}$ or $St(G,B_1,s_1)$ consisting of $q+2$ blocks with $B_1\cong H_{n-q(2k-1)-(2t+1),2k+1,k}$, $q$ blocks being  $K_{2k}$ and one block being  $K_{2t+2}$), where $s = k+ q(k-1)+t$,  $0\le t\le k-2$ and the graph $St(G,B_1,s_1)$ would be defined in Section 2.

\end{theorem}

Note that $\tau_{k,2}=2k$. Thus $h(2,k,s)={k+1\choose 2}-k(k+1)$.  We have the following result by Theorem
\ref{thm: exodd when n large enough}.

\begin{cor}
    If $k\geq 2$ and $s\geq 2k+1$, then
$$ ex(n,\{C_{\geq 2k+1},M_{s+1}\})={k\choose 2}+k(n-k)$$
when $n$ is large enough and the extremal graph is $ H_{n,2k+1,k}$.
\end{cor}

For the case forbidding cycles of length at least $2k$,  we define a function for integers $x,y$, and $z$ as following.
\begin{equation*}
    \begin{aligned}
    g(x,y,z):=& x{2k-1\choose r}+y{2k-2\choose r}+{z\choose r}+{k+1\choose r}\\
        -&(k+1+x(2k-2)+y(2k-3)+(z-1)){k-1\choose r-1}.
    \end{aligned}
\end{equation*}
Let $$\mathscr{T}_1:=\left\{(x,y,z)|x,y,z \in \mathbb{N}, 1\leq z\leq 2k-1,(k-1)x+(k-2)y+\lf\frac{z-1}{2}\rf+k\leq s\right\},$$
    $$\mathscr{T}_2:=\left\{(x,y,z)|x,y,z \in \mathbb{N}, 1\leq z\leq 2k-1,(k-1)x+(k-2)y+\lf\frac{z-1}{2}\rf+k-1\leq s\right\}.$$

\begin{theorem}\label{thm: exeven when n large enough}
    If $k\geq r$, $r\geq 2$ and $k\geq 3$, then we have
       $$ \excevenm={k-1\choose r-1}n+\max\left\{\max_{(x,y,z)\in \mathscr{T}_1} g(x,y,z), \max_{(x,y,z)\in \mathscr{T}_2}g(x,y,z)-{k-1\choose r-2}\right \}.$$
\end{theorem}
When $r=2$, we have the following result.
\begin{theorem}\label{cor: ex(c2k)}
    If $k\geq 2$, and $s=q(k-1)+t\geq k-1$, where $q=\lf\frac{s}{k-1}\rf\geq 1$ and $0\leq t\leq k-2$, then
    $$ex(n,\{C_{\geq 2k},M_{s+1}\})=(k-1)n-{k\choose 2}+(k-1)(q-1)+\epsilon$$ when $n$ is large enough, where  $\epsilon=1$ if $t\ge 1$ and 0 otherwise. Moreover, when $\epsilon=0$ (resp. $\epsilon=1$), the extremal graph is $St^1(n,2k,q)$ (resp. $St^2(n,2k,q)$), where $q=\lf\frac{s}{k-1}\rf$ and $St^1(n,2k,q)$ and $St^2(n,2k,q)$ would be given in Section 4.
\end{theorem}

The paper is organized as follows. In Section \ref{sec: pre}, we give basic definitions and lemmas.
In Section \ref{sec: exodd}, we will prove  Theorem \ref{thm: exodd when n large enough}. The proofs of Theorems \ref{thm: exeven when n large enough} and  \ref{cor: ex(c2k)} will be given in Section \ref{sec: exeven}.
\section{Preliminary}\label{sec: pre}
Let $G=(V,E)$ be a graph. For  $X\subseteq V$, we denote $N_X^r(v)=\{U:U\subseteq X,|U|=r-1,G[U\cup \{v\}]=K_r\}$ and $d_X^r(v)=|N_X^r(v)|$. If $X=V(G)$, we write $d^r(v)$ and $N^r(v)$ instead $d_{V(G)}^r(v)$ and $N_{V(G)}^r(v)$. And if $r=2$, we write $N_{X}(v)$ and $d_X(v)$ instead of $N_{X}^2(v)$ and $d_X^2(v)$, respectively. For $e\in E(G)$ (resp. $e\notin E(G)$), $G - e$ (resp. $G + e$) is the graph obtained by deleting  $e$  (resp. adding  $e$) in $G$. We will use $\nu(G)$ to denote the matching number of $G$.

Let $H$ be a connected graph. We  can decompose it into blocks, denoted as $B_i$, $i=1,\ldots, h$, and satisfying
$B_i\cap(\cup_{j=1}^{i-1}B_j)=\{u_i\}\subseteq V(B_i)$ ($2\le i\le h$), i.e. $u_i$ is a vertex cut separating $B_i$ and $\cup_{j=1}^{i-1}B_j$ (see Figure 1). If $H$ is 2-connected, then we let $H=B_1$.  We arbitrarily choose a vertex $u_1\in V(B_1)$. Let $H_1:=H$ and $H_i:=H_{i-1}-\sum_{v\in V(B_i)}vu_i+\sum_{v\in V(B_i)}vu_1$ for $i=2,\ldots,h$.
Then  the \textbf{star graph }of $H$, denoted by $St(H,B_1,u_1)$, is the graph  $H_h$.  Notice that the different choices of $B_1$ and $u_1\in V(B_1)$ may cause different star graphs (see Figure 2). The matchings of $H$ and $St(H,B_1,u_1)$ have the following relationship.
\begin{figure}[h]
        \centering        \includegraphics[width=0.6\textwidth]{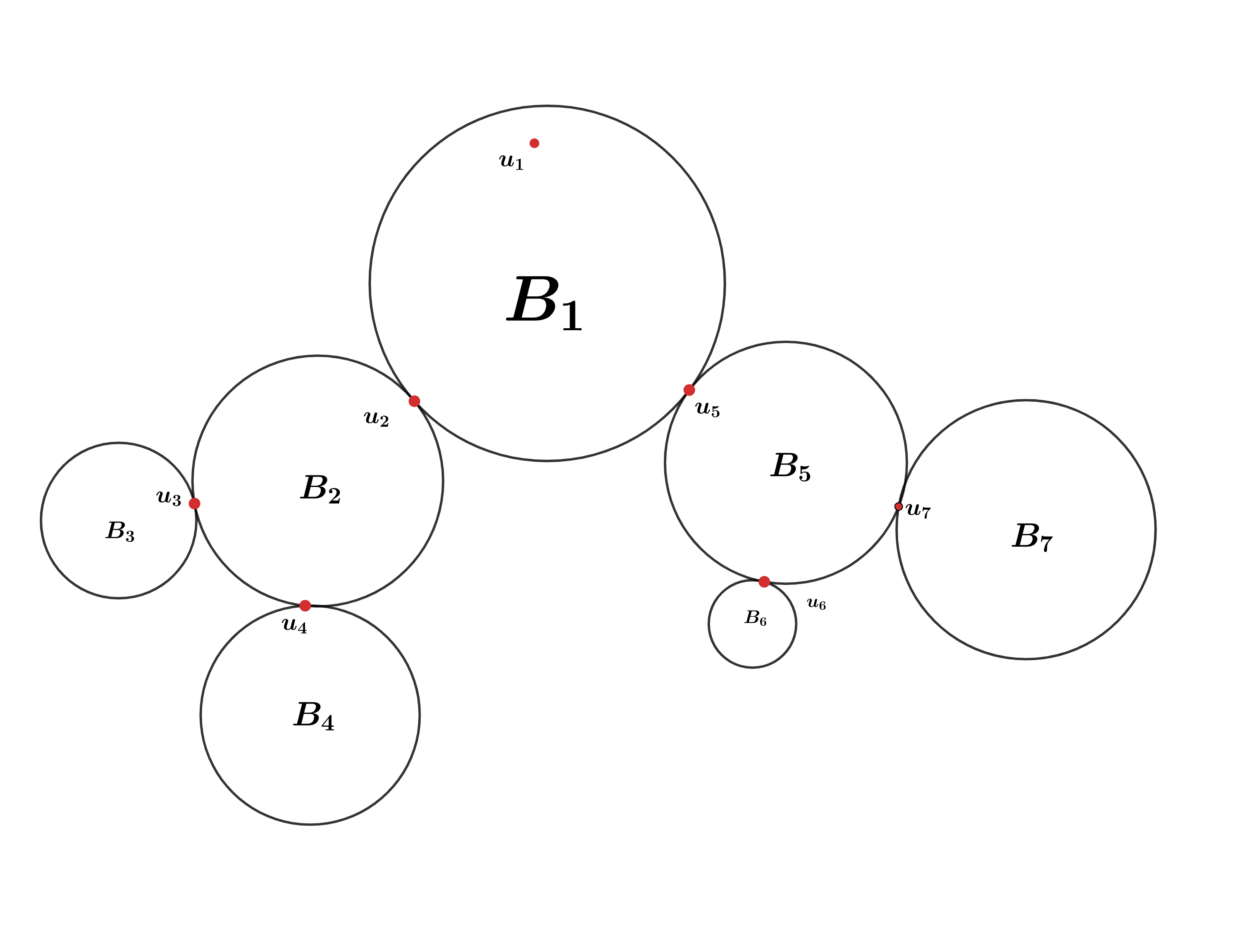}
        \caption{The block decomposition of  $G$ and its representative vertices.}
\end{figure}
\begin{figure}[h]\label{fig: Star graphs}
        \centering
\includegraphics[width=0.8\textwidth]{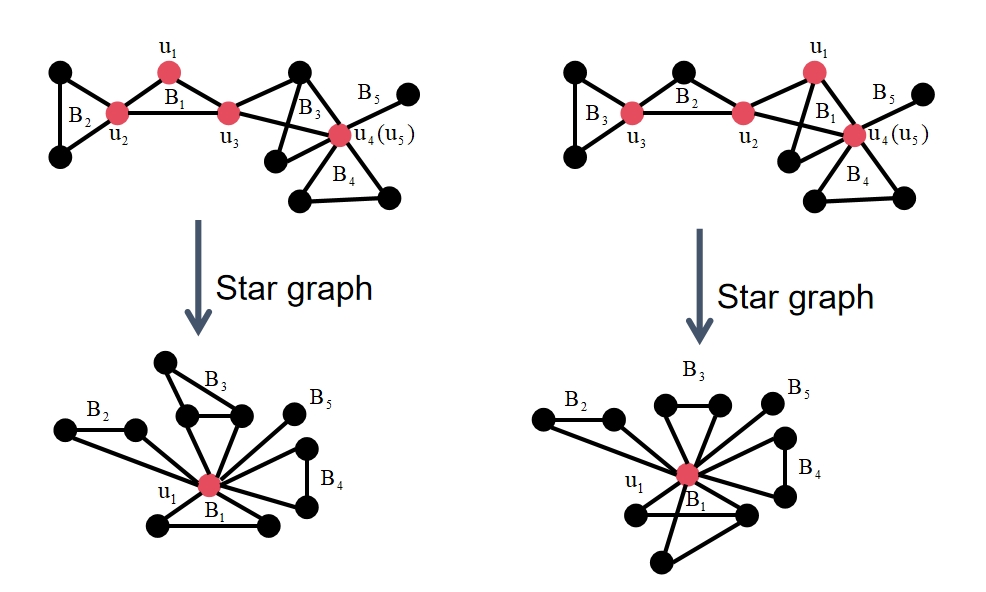}
        \caption{ Different choice of $B_1$ and $u_1$ may cause different star graph.}
\end{figure}
\begin{lemma}\label{lem: the matching number of star graph}
    Let $H$ be a connected graph and $St(H,B_1,u_1)$ be a star graph of $H$. Let $M$ be a matching of $St(H,B_1,u_1)$. If $u_1$ is not covered by $M$ or $u_1v\in M$ with $v\in V(B_1)$, then there is a matching with the same size in $H$.
\end{lemma}
\begin{proof} In both cases, we have $M\subseteq E(H)$ by the construction of $St(H,B_1,u_1)$. Thus the result holds.
\end{proof}

We introduce a few theorems about the long cycles and long paths in a graph, which we will use later.
\begin{theorem}[Dirac \cite{dirac1952some}]\label{thm: degree geq k, vertices geq 2k}
If the degree of every vertex of the $2$-connected graph $G$ is $\geq k~(k\geq 2)$ and if $v(G)\geq 2k$, then $G$ contains a cycle with at least $2k$ edges.
\end{theorem}
\begin{theorem}[Erd\H{o}s and Galli \cite{gallai1959maximal}]\label{thm: deg geq k but a single vertex, v leq 2k-1}
    If the vertex $u$ of the graph $G$ is not isolated and the degree of every vertex of $G$ distinct from $u$ is $\geq k$, $k\geq 2,$ and if $v(G)\leq 2k-1$, then $u$ is connected by a Hamiltonian path.
\end{theorem}

We need the following result about a graph with bounded matching number.
\begin{lemma}[Luo et al.  \cite{luo2024turannumbercompletebipartite}]\label{thm: TB theorem}
     Let $G$ be a graph with $n$ vertices. Then  $\nu(G)\le s$ if and only if there is a subset $X\subseteq V(G)$ such that
     \begin{equation}\label{eq: TB inequaulity}
         |X|+\sum_{i=1}^m \left\lfloor\frac{|V(C_i)|}{2}\right\rfloor\leq s,
     \end{equation}
      where $C_1,\ldots,C_m$ are  the components of $G-X $.
\end{lemma}

For every vertex set $X$ satisfies Lemma \ref{thm: TB theorem},  we denote $J_X(G):= \{i~|~ |V(C_i)| \geq 2, 1\leq i\leq m\}$ and  $I_X(G)=\{v~|~ v~ \text{is an isolated vertex in}~ G-X\}$, where $C_1,\ldots, C_m$ are the components of $G-X$. Then we have the following result.
\begin{lemma}\label{lem: almost all vertices independent}
    We have $|X|+\sum_{i\in J_X(G)}|V(C_i)|\leq 3s$, and $|I_X(G)|\geq n-3s$.
\end{lemma}
\begin{proof}
    For every $i\in J_X(G)$, we may assume that $|V(C_i)|=c_i$ is odd, i.e. there exists integer $k_i\geq1$ such that $c_i=2k_i+1$. Let $|X|=x$. Then (\ref{eq: TB inequaulity}) can be rewritten as $$x+\sum_{i\in J_X(G)}\lf\frac{2k_i+1}{2}\rf=x+\sum_{i\in J_X(G)}k_i\leq s.$$
    Then we have $\sum_{i\in J_X(G)}k_i\leq s-x$, and
    $$|X|+\sum_{i\in J_X(G)}|V(C_i)|\leq x+3\sum_{i\in J_X(G)}k_i\leq x+3(s-x)\leq 3s.$$
    Thus  $|I_X(G)|=n-|X|-\sum_{i\in J_X(G) }|V(C_i)|\geq n-3s.$
\end{proof}

Now let $G\in Ex(n,K_r,\{C_{\ge a},M_{s+1}\})$ and $a_0=\lfloor\frac{a-1}{2}\rfloor$ with $a\ge 5$, $a_0\ge r\ge 2$ and $s\ge \lceil\frac{a}{2}\rceil$.  In the following discussion, we assume that $n$ is large enough. Since $H_{n,a,a_0}$ is $\{C_{\ge a},M_{s+1}\}$-free when $s\ge \lceil\frac{a}{2}\rceil$, we have
\begin{equation}\label{eq1: lower bound of excoddm}
    N_r(G)\geq  N_r(H_{n,a,a_0})={a-a_0\choose r}+(n-a+a_0){a_0\choose r-1}.
\end{equation}
Let $X$ be a vertex set satisfying (\ref{eq: TB inequaulity}) in Lemma \ref{thm: TB theorem}.
We have the following results.

\begin{lemma}\label{lem: the main block in G}     There exist $S\subseteq X$ with $|S|=a_0$ and $Y\subseteq I_X(G)$ with $|Y|=\Theta(n)$ such that $G[S\cup \{y\}]=K_{a_0+1}$ for any $y\in Y$.
\end{lemma}
\begin{proof}
        By Lemma \ref{lem: almost all vertices independent},  the number of copies of $K_r$ in $G$ which do not contain any vertices in $I_X(G)$ is $O_{r,s}(1)$. From (\ref{eq1: lower bound of excoddm}), the number of copies of $K_r$ in $G$ is at least ${a_0\choose r-1}n$. So
        there is $Y\subseteq I_X(G)$ with $|Y|=\Theta(n)$ such that for any $y\in Y$, there is $S_y\subseteq X$ with $|S_y|\ge a_0$ and $G[S_y\cup \{y\}]$ is a clique when $n$ is large enough. For every $t\geq a_0+1$, if there exist more than $t$ vertices connecting $t$ vertices of $X$, then we can find a cycle of length more than $2t\ge a$, a contradiction. Thus the number of vertices in $I_X(G)$ connecting more than $a_0$ vertices of $X$ is $O_{s,k}(1).$

    So there are $\Theta(n)$ vertices in $I_X(G)$ connecting a clique in $X$ with size exactly $a_0$. Since there are at most ${|X|\choose a_0}$ subsets of $X$ with size $a_0$, according to the Pigeonhole Principal, we are done.
\end{proof}

Let $S_0$ be the subset given by Lemma \ref{lem: the main block in G} in the following discussion. If $G$ has an isolated vertex, say $x$, then let $G'=G+\sum_{s\in S_0}xs$. If $G$ have two components which have at least two vertices, say $G_1$ and $G_2$, let $G'=G-\sum_{x\in N_{G_1}(x_1)}x_1x+\sum_{x\in N_{G_1}(x_1)}x_2x+\sum_{s\in S_0}x_1s$, where $x_i\in V(G_i)$ for $i=1,2$ and $S_0\cap V(G_1)=\emptyset$. In each case, we have $G'$ is $\{C_{\ge a},M_{s+1}\}$-free, but $N_r(G')>N_r(G)$, a contradiction with $G\in Ex(n,K_r,\{C_{\ge a},M_{s+1}\})$. So  we have that $G$ is connected.

We choose  the block containing $S_0$  as the first block $B_1$ in the decomposition of $G$ to construct the star graph $St(G,B_1,s_1)$ of $G$, where
 $s_1\in S_0$. Then  $\nu(B_1)\ge a_0$.
Also, $s_1$ is covered by each maximum matching of $St(G,B_1,s_1)$.

\begin{lemma}\label{lem: G=St(G)} We have
     $St(G,B_1,s_1)\in Ex(n,K_r,\{C_{\ge a},M_{s+1}\})$.
\end{lemma}
\begin{proof}
   First, we claim that  $\nu(St(G,B_1,s_1))\le \nu(G)$. By Lemma \ref{lem: the matching number of star graph}, we only need to consider the case that there exists a maximum matching $M$  of $St(G,B_1,s_1)$, such that  $s_1u\in M$ and $u\not\in V(B_1)$. By Lemma \ref{lem: the main block in G}, we can choose a vertex $x\in I_X(G)\cap V(B_1)$ such that $(M\setminus\{s_1u\})\cup\{s_1x\}$ is a maximum matching of $St(G,B_1,s_1)$ and we are done by Lemma \ref{lem: the matching number of star graph}.

    It is easy to check that  $St(G,B_1,s_1)$ is $C_{\ge a}$-free and $N_r(G)=N_r(St(G,B_1,s_1))$. So we finish the proof.
\end{proof}

By Lemma \ref{lem: G=St(G)},  we will assume $G=St(G,B_1,s_1)$ in the next discussion. Then  $s_1$ is a vertex cut of $G$ separating $G$ into blocks $B_i$, $i=1,\ldots,h$ if $V(G)\setminus V(B_1)\not=\emptyset$. Let $A\subseteq V(G)$ and $v\notin A$. We define $G[v\to A]$, call \textbf{switch $v$ to $A$}, to be the
graph obtained from $G$ by deleting all edges joining with $v$ and adding new edges connecting  $v$ with every vertex in $A$.
Note that $G[v\to S_0]$ is still $\{C_{\ge a},M_{s+1}\}$-free if $v\in V(G)\setminus V(B_1)$. If there is $v\in V(G)\setminus V(B_1)$ such that $d^r(v)\le {a_0\choose r-1}$, then
$$N_r(G[v \to S_0])=N_r(G)-d^r(v)+{a_0\choose r-1}\ge N_r(G).$$ So we will assume $d^r(v)\ge {a_0\choose r-1} +1$ for any $v\in V(G)\setminus V(B_1)$ if $V(G)\setminus V(B_1)\not=\emptyset$.  Hence $I_X(G)\subseteq V(B_1)$ and if $V(G)\setminus V(B_1)\not=\emptyset$, we have $d_{B_i}(v)=d(v)\geq a_0+1$ for each $v\in V(G)\setminus V(B_1)$ by $a_0\ge r\ge 2$.
Recall $a_0=\lfloor\frac{a-1}{2}\rfloor$.

\begin{lemma}\label{lem: no block with size geq 2k+1}
  If $V(G)\setminus V(B_1)\not=\emptyset$, then for any $i\geq 2$,  $a_0+2\le v(B_i)\le a-1$.
\end{lemma}
\begin{proof} Since $d_{B_i}(v)=d(v)\geq a_0+1$ for each $v\in V(G)\setminus V(B_1)$, we have $v(B_i)\ge a_0+2$.

     Suppose there is $i$, say $i=2$, such that $v(B_2)\ge a$. Recall for every $v\in V(B_2)\setminus\{s_1\}$,  $d_{B_2}(v)=d(v)\geq a_0+1$.
    By Theorem \ref{thm: degree geq k, vertices geq 2k} and $G$ being $C_{\ge a}$-free, we have $|V(B_2)\setminus\{s_1\}|=a-1$ and $G[V(B_2)\setminus\{s_1\}]$ has a Hamiltonian cycle. So $v(B_2)=a$ and the number of the maximum matching of $G$ contained in $B_2$ is  $a_0$ because $s_1$ is covered by each maximum matching in $B_1$. Since $G$ is $\{C_{\ge a},M_{s+1}\}$-free, the circumference of $B_2$ is $a-1$.
    Then by Theorem \ref{thm: 2-connected without Ck, count Kr}, we have that $$N_r(B_2)\leq \max\left\{{a-2\choose r}+2{2\choose r-1},{a-a_0\choose r}+a_0{a_0\choose r-1}\right\},$$
    and the extremal graph is either $H_{a,a,2}$ or $H_{a,a,a_0}$. Note that both of these graphs are $C_{\ge a}$-free with matching number  $a_0$, so we may assume that $B_2$ is isomorphic to one of the extremal graphs. But in these two graphs, there are at least two vertices with degrees less than $a_0+1$, a contradiction with our assumption.
\end{proof}

\begin{lemma}\label{lem: all blocks are clique}
   If $V(G)\setminus V(B_1)\not=\emptyset$, we may assume $B_i$ is a clique  for any $2\leq i\leq h$. Moreover $v(B_i)$ is even if $v(B_i)\le a-2$ for any $2\leq i\leq h$.
\end{lemma}
\begin{proof}
Note that $d_{B_i}(v)=d(v)\geq a_0+1$ for each $v\in V(G)\setminus V(B_1)$. By Lemma \ref{lem: no block with size geq 2k+1}, $v(B_i)\le a-1\le 2a_0-1$. By Theorem \ref{thm: deg geq k but a single vertex, v leq 2k-1}, there exists a Hamiltonian path connecting $s_1$. Since $s_1$ is covered by each maximum matching in $B_1$,
      $\nu(B_i)= \lf (v(B_i)-1)/2\rf$ for $2\le i\le h$. So we can replace $B_i$ with a complete graph of the same size, and the number of $K_r$ will not decrease.

      Suppose there is $i$, say $i=2$, such that $v(B_2)$ is odd and $v(B_2)\le a-2$. Let $v\in I_X(G)$. Note that $s_1$ is covered by each maximum matching in $B_1$. Then $G[v\to B_2]$ is $\{C_{\ge a},M_{s+1}\}$-free. But
    $$N_r(G[v\to B_2])-N_r(G)={v(B_2)+1\choose r}-{v(B_2)\choose r}-{a_0\choose r-1}={v(B_2)\choose r-1}-{a_0\choose r-1}>0$$ by $v(B_2)\ge a_0+2$, a contradiction.
\end{proof}


 \begin{lemma}\label{new2}If $V(G)\setminus V(B_1)\not=\emptyset$, then
  $|\{i ~|~  a_0+1\le v(B_i)\le a-2-\varepsilon,2\le i\le h\}|\le 1$, where $\varepsilon=0$ if $a$ is odd; else $\varepsilon=1$.
\end{lemma}

\begin{proof} Suppose there are $i,j\in\{2,\ldots,h\}$, say $i=2$ and $j=3$, such that $a_0+1\le v(B_i)\le a-2-\varepsilon$ for $i=2,3$. Assume $v(B_2)\le v(B_3)$. By Lemma \ref{lem: all blocks are clique}, we have  $v(B_3)\le a-3-\varepsilon$.
Let $v_1,v_2\in V(B_2)\setminus \{s_1\}$. Then we switch $v_1$ and $v_2$ to $B_3$ respectively and add an edge connecting $v_1$ and $v_2$. We denote the obtained graph by $G'$. Then $G'$ is $\{C_{\ge a},M_{s+1}\}$-free. Since
     $$           {v(B_2)\choose r}+{v(B_3)\choose r}< {v(B_2)-2\choose r}+{v(B_3)+2\choose r},
       $$
we have $N_r(G')>N_r(G)$, a contradiction.
\end{proof}

\section{Proof of Theorem \ref{thm: exodd when n large enough}}\label{sec: exodd}

In this Section, we prove Theorem \ref{thm: exodd when n large enough}.
We will use the same notations (for example, $S_0$, $B_1,\ldots,B_h$, $I_X(G)$, etc) as that in Section 2. Then $a=2k+1$ and $a_0=k$.   First, we prove the upper bound of $\excoddm$  and the construction of the lower bound will be given in the proof.

Let $G\in Ex(n,K_r,\{C_{\ge 2k+1},M_{s+1}\})$. Moreover, we assume $G$ has the least number of blocks among all graphs in
$Ex(n,K_r,\{C_{\ge 2k+1},M_{s+1}\})$. By the discussion in Section 2, we have $G=St(G,B_1,s_1)$, and for any $i\geq 2$, $B_i$ is a clique, $k+2\le v(B_i)\le 2k$ and $d_{B_i}(v)=d(v)\geq k+1$ for every $v\in V(B_i)\setminus \{s_1\}$ if $V(G)\setminus V(B_1)\not=\emptyset$.

Note that $B_1$ is 2-connected and $\{C_{\ge 2k+1},M_{s+1}\}$-free.
According to Corollary \ref{cor: ex(n,K_r) for 2-connected graph}, $N_r(B_1)\leq f_r(v(B_1),2k+1,k)$ and the equality holds if and only if $B_1\cong H_{v(B_1),2k+1,k}$. Since   $\nu(B_1)\ge k$ and  $\nu(H_{v(B_1),2k+1,k})=k$, we may assume $B_1\cong H_{v(B_1),2k+1,k}$. Then  $\nu(B_1)=k$ and
$$N_r(B_1)={k+1\choose r}+(v(B_1)-k-1){k\choose r-1}.$$

     Recall $\tau_{k,r}=\min\{k_0~|~k_0{k\choose r-1}< {k_0+1\choose r}\}$. We have the following result.

\begin{lemma}\label{new}
     If $2k\leq \tau_{k,r}$, then $h=1$ (i.e.,  $V(G)\setminus V(B_1)=\emptyset$). More over $N_r(G)\leq {k+1\choose r}+(n-k-1){k\choose r-1}$ and the equality holds if and only if  $ G\cong H_{n,2k+1,k}$.
\end{lemma}
\begin{proof} Suppose $h\ge 2$. Now we switch all vertices in $V(B_2)\setminus \{s_1\}$ to $S_0$ one by one. Denote the new graph by $G'$. Then $G'$ is  $\{C_{\ge 2k+1},M_{s+1}\}$-free and$$N_r(G')-N_r(G)=(v(B_2)-1){k\choose r-1}-{(v(B_2)-1)+1\choose r}.$$ By Lemma \ref{lem: no block with size geq 2k+1}, $ v(B_2)-1\le 2k-1<\tau_{k,r}$. Hence $N_r(G')-N_r(G)\geq 0$, but $G'$ has less blocks than $G$, a contradiction. So $h=1$ and then $G=B_1$. Thus the result holds.
\end{proof}

Now we are ready to prove Theorem \ref{thm: exodd when n large enough}.

\begin{reproof} By (the proof of) Lemma \ref{new}, we can assume $h\ge 2$ and $2k-1\geq \tau_{k,r}$.
    By Lemma \ref{lem: no block with size geq 2k+1}, $k+2\le v(B_i)\leq 2k$ for every $2\leq i\leq h$. By Lemma \ref{lem: all blocks are clique},  for any $2\leq i\leq h$, $B_i$ is a clique.

     By Lemma \ref{new2}, we assume $v(B_i)=2k$ for all $2\le i\le h-1$ and $k+2\le v(B_h)\le 2k$.
 Let $s = k+ q(k-1)+t$, where $0\le t\le k-2$. Since $\nu(G)=k+(h-2)(k-1)+\lf\frac{v(B_h)-1}{2}\rf$, we have
  \begin{equation}\label{matching}
  k+(h-2)(k-1)+\lf\frac{v(B_h)-1}{2}\rf\le k+ q(k-1)+t.
\end{equation}
 We finish the proof by considering the following two cases.

\begin{mycase}
    \case$ 2t+1<\tau_{k,r}$.

 We first prove that $h-2\leq q-1$. By (\ref{matching}), $h-2\le q$ and we just need to consider the case $t>0$. Suppose $h-2=q$. By (\ref{matching}), we have $v(B_h)\leq 2t+2\leq 2k-2$. We switch all vertices in $V(B_h)\setminus \{s_1\}$ to $S_0$ one by one. Denote the new graph by $G'$. Then $G'$ is $C_{\geq k}$-free  and  $\nu(G')=k+q(k-1)< s$. But
    $$N_r(G')-N_r(G)= (v(B_h)-1){k\choose r-1}-{v(B_h)\choose r}.$$
    Since $v(B_h)-1\leq 2t+1<\tau_{k,r}$,  we have $N_r(G')-N(G)\geq 0$ and $G'$ has less number of blocks than $G$, a contradiction. So $h-2\leq q-1$ holds.

     Then we claim that $v(B_h)=2k$. Suppose $v(B_h)\le 2k-1$. By Lemma \ref{lem: all blocks are clique}, $v(B_h)$ is even. Then $G[v\to B_h]$ with $v\in I_x(G)$ is $\{C_{2k+1},M_{s+1}\}$-free and
      $$N_r(G')-N_r(G)={v(B_h)\choose r-1}-{k\choose r-1}>0,$$  a contradiction. Hence $v(B_h)=2k$ and then
     we  have
    $$N_r(G)\leq q{2k\choose r}+{k+1\choose r}+(n-k-1-q(2k-1)){k\choose r-1}. $$
    The equality holds if and only if $h-1=q$. The extremal graph is $St(G,B_1,s_1)$ consisting of $q+1$ blocks, with $B_1\cong H_{n-q(2k-1),2k+1,k}$ and $q$ blocks being cliques with order $2k$.

    \case $2t+1\geq \tau_{k,r}.$

   By (\ref{matching}), we have $h-2\leq q.$
        We claim that $h-2=q$. Suppose $h-2\leq q-1$. Then we may assume $v(B_h)=2k$  otherwise we can switch $v\in I_X(G)$ to $B_h$ and get a graph with more copies of $K_r$ by the same argument as Case 1. Now $\nu(G)=k+q(k-1)$. Let $G'$ be the graph obtained from $G$ by deleting $2t+1$ vertices from $I_X(G)$ and adding a clique of order $2t+2$ that intersects $G$ with  $s_1.$ Then  $G'$ is still $C_{\geq k}$-free and $\nu(G')=k+(h-1)(k-1)+t\leq s$. Moreover,
    $$N_r(G')-N_r(G)={2t+2\choose r}-(2t+1){k\choose r-1}>0,$$
     a contradiction. Now $h-2=q$ which implies $v(B_h)\leq 2t+2$ and
    \begin{equation*}
        \begin{aligned}
            N_r(G)=& q{2k\choose r}+{v(B_h)\choose r}+{k+1\choose r}
            +(n-q(2k-1)-(v(B_h)-1)-k-1){k\choose r-1}\\
            \leq & q{2k\choose r}+{2t+2\choose r}+{k+1\choose r}+(n-q(2k-1)-2t-1-k-1){k\choose r-1}.
        \end{aligned}
    \end{equation*} The last equality holds by (\ref{eq: the convexity of k choose r}) and $2t+1\geq \tau_{k,r}$.
    The equality holds if and only if $v(B_h)= 2t+2$.
The extremal graph is $St(G,B_1,s_1)$ consisting of $q+2$ blocks with $B_1\cong H_{n-q(2k-1)-(2t+1),2k+1,k}$, $q$ blocks being  $K_{2k}$ and one block being  $K_{2t+2}$. Now we have finished the proof.
\end{mycase}
\end{reproof}

\section{Proof of Theorems \ref{thm: exeven when n large enough} and \ref{cor: ex(c2k)}}\label{sec: exeven}

In this Section, we prove Theorems \ref{thm: exeven when n large enough} and \ref{cor: ex(c2k)}.
We will use the same notations (for example, $S_0$, $B_1,\ldots,B_h$, $I_X(G)$, etc) as that in Section 2. Then $a=2k$ and $a_0=k-1$.   First, we prove the upper bound of $ex(n,K_r,\{C_{\ge 2k},M_{s+1}\})$  and the construction of the lower bound will be given in the proof.

Let $G\in Ex(n,K_r,\{C_{\ge 2k},M_{s+1}\})$.  By the discussion in Section 2, we have $G=St(G,B_1,s_1)$, and for any $i\geq 2$, $B_i$ is a clique, $k+2\le v(B_i)\le 2k-1$ and $d_{B_i}(v)=d(v)\geq k$ for every $v\in V(B_i)\setminus \{s_1\}$ if $V(G)\setminus V(B_1)\not=\emptyset$. Also $\nu(B_1)\ge k-1$.



  By Lemma \ref{new2}, we assume there are $x$ blocks of order $2k-1$, $y$ blocks of order $2k-2$ in $\{B_2,\ldots,B_{h-1}\}$ and  let $v(B_h)=z$.   Then $v(B_1)= n-x(2k-2)-y(2k-3)-(z-1)$.
  By Lemma \ref{lem: no block with size geq 2k+1}, we have $k+1\le z\le 2k-1$. Since our result is about a maximum problem,   we let $z\ge 1$. Then $z=1$ implies that  $B_h$ does not exist.

  \begin{reprooff}

  Note that $\nu(B_1)\ge k-1$. We complete the proof by considering two cases.


\begin{mycase}
    \case  $\nu(B_1)\ge k$.
 Since $B_1$ is 2-connected and $\{C_{\ge 2k},M_{s+1}\}$-free,
by Corollary \ref{cor: ex(n,K_r) for 2-connected graph}, $N_r(B_1)\leq f_r(v(B_1),2k,k-1)$ and the equality holds if and only if $B_1\cong H_{v(B_1),2k,k-1}$. Since $\nu(B_1)\ge k$ and  $\nu(H_{v(B_1),2k+1,k})=k$, we may assume $B_1\cong H_{v(B_1),2k,k-1}$. Then  $\nu(B_1)=k$ and
$$N_r(B_1)={k+1\choose r}+(v(B_1)-k-1){k-1\choose r-1}.$$

   Since $\nu(G)\le s$, we have that $k+x(k-1)+y(k-2)+\lf\frac{z-1}{2}\rf\leq s$.
  Above all, we have that
\begin{equation*}
    \begin{aligned}
    g_1(x,y,z,n):=N_r(G)=& x{2k-1\choose r}+y{2k-2\choose r}+{z\choose r}+{k+1\choose r}\\
        +&(n-x(2k-2)-y(2k-3)-(z-1)-k-1){k-1\choose r-1}.
    \end{aligned}
\end{equation*}
Note that $g_1(x,y,z,n)={k-1\choose r-1}n+g(x,y,z)$.
The optimal value of $g_1(x,y,z,n)$ under the constriction $k+x(k-1)+y(k-2)+\lf\frac{z-1}{2}\rf\leq s$ and $1\leq z\leq 2k-1$ is denoted as $M_1$. Every solution of $(x,y,z)$ responding to a  $\{C_{\geq 2k},M_{s+1}\}$-free graph.

\case  $\nu(B_1)=k-1$.
In this case, we have $\delta(B_1)\leq k-1$; otherwise, we can find a cycle of length at least $2k,$ a contradiction with $G$ being $C_{\ge 2k}$-free. By Theorem \ref{thm: duan thm}, we may assume $B_1\cong H_{v(B_1),2k-1,k-1}$ and then $N_r(B_1)={k+1\choose r}-{k-1\choose r-2}+(v(B_1)-k-1){k-1\choose r-1}$.
Since $\nu(G)\le s$, we have that $k-1+x(k-1)+y(k-2)+\lf\frac{z-1}{2}\rf\leq s$.  Above all, we have that
\begin{equation*}
    \begin{aligned}
    g_2(x,y,z,n):=N_r(G)=& x{2k-1\choose r}+y{2k-2\choose r}+{z\choose r}+{k+1\choose r}-{k-1\choose r-2}\\
        +&(n-x(2k-2)-y(2k-3)-(z-1)-k-1){k-1\choose r-1}.
    \end{aligned}
\end{equation*}
Note that $g_2(x,y,z,n)={k-1\choose r-1}n+g(x,y,z)-{k-1\choose r-2}$. The optimal value of $g_2(x,y,z,n)$ under the constriction $k-1+x(k-1)+y(k-2)+\lf\frac{z-1}{2}\rf\leq s$ and $1\leq z\leq 2k-1$ is denoted as $M_2$. Every solution of $(x,y,z)$ responding to a  $\{C_{\geq 2k},M_{s+1}\}$-free graph.
\end{mycase}
Then we have $N_r(G)\leq \max\{M_1,M_2\}$ and we are done.
\end{reprooff}

Now we will use Theorem \ref{thm: exeven when n large enough} to prove Theorem \ref{cor: ex(c2k)}.

\begin{reproofcor}
    Let $G\in Ex(n,K_2,\{C_{\ge 2k},M_{s+1}\})$. We will use the same notations as that in the proof of Theorem \ref{thm: exeven when n large enough}. Then we have $G=St(G,B_1,s_1)$, and for any $i\geq 2$, $B_i$ is a clique, $k+2\le v(B_i)\le 2k-1$ and $d_{B_i}(v)=d(v)\geq k$ for every $v\in V(B_i)\setminus \{s_1\}$ if $V(G)\setminus V(B_1)\not=\emptyset$. Also $\nu(B_1)\ge k-1$. We first have the following claim.
    \begin{claim}\label{claim}
        We may assume that $ v(B_i)= 2k-1$ for every $i\geq 2$ if $V(G)\setminus V(B_1)\not=\emptyset$.
    \end{claim}
\begin{proof}
    Suppose there exist $i$, say $i=h$, such that  $v(B_h)=b\leq 2k-2$. Then we can switch all vertices in $V(B_h)\setminus\{s_1\}$ to $S_0$ one by one. We denote the obtained graph by $G'$. Then $G'$ is $\{C_{\ge 2k},M_{s+1}\}$-free. Since
    $$e(G')-e(G)=(k-1)(b-1)-{b\choose2} = (b-1)\left(k-1-\frac{b}{2}\right)\geq 0,$$
    we are done.
\end{proof}

By Claim \ref{claim}, we have $y=0$ and $z=1$. Let $s=q(k-1)+t\geq k-1$, where $q=\lf\frac{s}{k-1}\rf$ and $0\leq t\leq k-2$. We denote two special star graphs.  Let $St^1(n,2k,q)$ (resp. $St^2(n,2k,q)$)  be a $St(G,B_1,s_1)$ consisting of $q$ blocks with $B_1\cong H_{n-(q-1)(2k-2),2k-1,k-1}$ (resp. $B_1\cong H_{n-(q-1)(2k-2),2k,k-1}$) and $q-1$ blocks being  $K_{2k-1}$.

If $\nu(B_1)=k-1$, then $e(G)=\max\{g_2(x,0,1,n):k-1+x(k-1)\le q(k-1)+t\}$ by the proof of  Theorem \ref{thm: exeven when n large enough}. It is not difficult to check that the optimal value is obtained when $x=q-1$. Then
\begin{equation}
    \begin{aligned}
        e(G)= & {k+1\choose 2}+(q-1){2k-1\choose 2}+ (k-1)(n-(q-1)(2k-2)-k-1)-1\\
        =& (k-1)n-{k\choose 2}+(k-1)(q-1),
    \end{aligned}
\end{equation}
and the equality holds if $G\cong St^1(n,2k,q)$.

If $\nu(B_1)=k$, then $e(G)=\max\{g_1(x,0,1,n):k+x(k-1)\le q(k-1)+t\}$ by the proof of  Theorem \ref{thm: exeven when n large enough}. It is not difficult to check that if $t>0$ (resp. $t=0$), then the optimal value is obtained when $x=q-1$ (resp. $x=q-2$).
If $t>0$, then
\begin{equation*}
    \begin{aligned}
        e(G)= & {k+1\choose 2}+(q-1){2k-1\choose 2}+ (k-1)(n-(q-1)(2k-2)-k-1)\\
        =& (k-1)n-{k\choose 2}+(k-1)(q-1)+1
    \end{aligned}
\end{equation*}
and the equality holds when $G\cong St^2(n,2k,q)$.
If $t=0$, then
\begin{equation*}
    \begin{aligned}
        e(G)= & {k+1\choose 2}+(q-2){2k-1\choose 2}+ (k-1)(n-(q-2)(2k-2)-k-1)\\
        =& (k-1)n-{k\choose 2}+(k-1)(q-2)<e(St^1(n,2k,q)).
    \end{aligned}
\end{equation*}
So $ex(n,\{C_{\geq 2k},M_{s+1}\})=(k-1)n-{k\choose 2}+(k-1)(q-1)+\epsilon, $ where  $\epsilon=1$ if $t\ge 1$, otherwise $\epsilon=0$
\end{reproofcor}

\section{Conclusion}
In this paper, we determined the exact value of $ex(n,\{C_{\geq a},M_{s+1}\})$. According to the proof above, the star graph plays an important role, and the extremal graphs can be constructed in a few patterns through the definition of the star graph.

\section*{Acknowledgement}
This work is  supported by  the National Natural Science Foundation of China~(Grant 12171272 \& 12426603).

\end{document}